\documentclass{amsart}
\usepackage{tikz}
\usepackage{hyperref}   
\usepackage{amssymb,xcolor}
\usepackage{stmaryrd} 

\usepackage{mathtools}  
\mathtoolsset{showonlyrefs}

\theoremstyle{plain}
\newtheorem{definition}{Definition}
\theoremstyle{plain}
\newtheorem{theorem}{Theorem}[section]
\theoremstyle{plain}

\theoremstyle{remark}
\newtheorem{remark}[theorem]{Remark}
\theoremstyle{plain}

\numberwithin{equation}{section}

\begin{document}

\title{Hypersurfaces of prescribed null expansion}

\author{Xiaoxiang Chai}
\address{Korea Institute for Advanced Study, Seoul 02455, South Korea}
\email{xxchai@kias.re.kr}

\begin{abstract}
  We study hypersurfaces with prescribed null expansion in an initial data
  set. We propose a notion of stability and prove a topology theorem.
  Eichmair's Perron approach toward the existence of marginally outer trapped
  surface adapts to the settings of hypersurfaces with prescribed null
  expansion with only minor modifications.
\end{abstract}

{\maketitle}

Gromov {\cite{gromov-metric-2018,gromov-four-2021}} studied stable
$\mu$-bubbles in the study of scalar curvature which leads to many interesting
results in positive scalar curvature. See {\cite{gromov-four-2021}}. The use
of stable $\mu$-bubble dates back to minimal surface techniques used by
Schoen-Yau settled the positive mass conjecture {\cite{schoen-proof-1979}} and
{\cite{schoen-existence-1979}}. The analog of minimal surfaces in initial data
sets is a \text{{\itshape{marginally outer trapped surface}}} (in short,
\text{{\itshape{MOTS}}}; See {\cite{andersson-local-2005}}). An initial data
set $(M, g, p)$ is a Riemannian manifold $(M, g)$ equipped with an extra
symmetric 2-form $p$. A MOTS arises boundaries of blow up sets of Jang
equation {\cite{schoen-proof-1981}}. It was applied to establish the spacetime
positive mass theorem {\cite{eichmair-spacetime-2016}}.

\begin{definition}
  Given a function $h$ on $(M, g)$, a hypersurface $\Sigma$ is called a
  hypersurface of prescribed null expansion if
  \begin{equation}
    \theta := H_{\Sigma} +\ensuremath{\operatorname{tr}}_{\Sigma} p = h
    \label{pne}
  \end{equation}
  where $h$ is a function on $(M, g)$ and $\theta$ is the null expansion.
\end{definition}

The case $h \equiv 0$ gives the definition of MOTS. Motivated by the stability
of MOTS {\cite{andersson-local-2005}} and the stability of the $\mu$-bubble,
we define the following.

\begin{definition}
  We say that $\Sigma$ is stable if there exists a vector field $X = \varphi
  \nu$ with nonzero $\varphi \geqslant 0$ such that
  \begin{equation}
    \delta_X (H +\ensuremath{\operatorname{tr}}_{\Sigma} p - h) \geqslant 0.
    \label{stability}
  \end{equation}
\end{definition}

It is desirable to find a physical interpretation of such hypersurfaces. By
{\cite[(10)]{andersson-jangs-2010}},
\begin{equation}
  \delta_X \theta = - \Delta \varphi + 2 \langle W, \nabla \varphi \rangle +
  (\ensuremath{\operatorname{div}}W - W + Q) \varphi - \tfrac{1}{2} \theta
  (\theta - 2\ensuremath{\operatorname{tr}}p)
\end{equation}
where $W$ is the vector field tangential to $\Sigma$ which is dual to $p (\nu,
\cdot)$ and
\begin{equation}
  Q = \tfrac{1}{2} R_{\Sigma} - \mu - J (\nu) - \tfrac{1}{2} |p_{\Sigma} +
  A|^2,
\end{equation}
and $\mu = \tfrac{1}{2} R_g - |p|_g^2 + (\ensuremath{\operatorname{tr}}_g
p)^2$, $J =\ensuremath{\operatorname{div}}_g p - \mathrm{d}
(\ensuremath{\operatorname{tr}}_g p)$, $A$ is the second fundamental form of
$\Sigma$ in $M$ and we write in short $\chi = p_{\Sigma} + A$ the null second
fundamental form. With $X (h) = \varphi \langle D h, \nu \rangle$ and $\theta
= h$, we have that
\begin{align}
& L \varphi \\
= & - \Delta \varphi + 2 \langle W, \nabla \varphi \rangle +
(\ensuremath{\operatorname{div}}W - W + Q) \varphi \\
& \quad - \tfrac{1}{2} (h^2 - 2 h\ensuremath{\operatorname{tr}}p + 2 \nu
(h)) \varphi .
\end{align}
Equivalently, $\Sigma$ is stable if the first eigenvalue to the eigen problem
$L \varphi = \lambda \varphi$ is real and nonnegative by Krein-Rutman theorem
(See {\cite{andersson-local-2005}}). Note that $h \equiv 0$ recovers the
stability operator of a MOTS and $p = 0$ recovers the stability operator of a
$\mu$-bubble.

\begin{remark}
  By requiring that $\partial M \neq \emptyset$ and $\Sigma$ does not change
  intersecting angle along the deformation induced by $X$, one can define a
  capillary version of \eqref{stability} (See {\cite{alaee-stable-2020}}).
  More generally, one can prescribe arbitrary angle on $\partial M$ as well,
  the boundary stability operator is no different from that of the case $p
  \equiv 0$. See {\cite{gromov-four-2021}} for a derivation of boundary
  operator for the case $p \equiv 0$.
\end{remark}

Assuming for some $h$ defined on $M$,
\begin{equation}
  \mu - |J| + \tfrac{1}{2} (\tfrac{n}{n - 1} h^2 - 2
  h\ensuremath{\operatorname{tr}}p - 2| D h|) \geqslant 0. \label{modfied dec}
\end{equation}
The condition \eqref{modfied dec} when $h \equiv 0$ is called
\text{{\itshape{dominant energy condition}}}. When $p \equiv 0$, it reduces to
the relaxed positivity of scalar curvature {\cite{gromov-four-2021}}. We have
the following theorem asserting the Yamabe type of a stable hypersurface of
null expansion $\theta = h$.

\begin{theorem}
  \label{topology}Assume that $\Sigma$ is a stable hypersurface of prescribed
  null expansion $h$ in an initial data set $(M, g, p)$ satisfying
  \eqref{modfied dec}, then $\Sigma$ is of positive Yamabe type unless
  $\Sigma$ is Ricci flat, $\chi^0$ vanishes and $\mu - |J| + \tfrac{1}{2}
  (\tfrac{n}{n - 1} h^2 - 2 h\ensuremath{\operatorname{tr}}p + 2 \nu (h))$
  vanishes along $\Sigma$.
\end{theorem}

\begin{proof}[Proof of Theorem \ref{topology}]
  Consider the eigenvalue problem
  \begin{equation}
    - \Delta \varphi + [Q - \tfrac{1}{2} (h^2 - 2
    h\ensuremath{\operatorname{tr}}p + 2 \nu (h))] \varphi = \lambda \varphi .
  \end{equation}
  Since $\Sigma$ is stable, it follows the same lines as in
  {\cite{galloway-generalization-2006}} that the first eigenvalue $\lambda_1
  \geqslant 0$ and the associated eigenfunction $f$ can be chosen positive
  everywhere on $\Sigma$. Consider the metric $\tilde{h} = f^{\tfrac{2}{n -
  2}} h$, we have that the scalar curvature $\tilde{S}$ of $\Sigma$ with
  respect to $\tilde{h}$ is
\begin{align}
& \tilde{S} \\
= & f^{- \tfrac{n}{n - 2}} (- 2 \Delta f + S f + \tfrac{n - 1}{n - 2} f^{-
1} | \nabla f|^2) \label{conformal related S} \\
= & f^{- \tfrac{2}{n - 2}} [2 \lambda_1 + 2 (\mu + J (\nu)) \\
& \quad + (h^2 - 2 h\ensuremath{\operatorname{tr}}p + 2 \langle D h, \nu
\rangle) + | \chi |^2 + \tfrac{n - 1}{n - 2} f^{- 1} | \nabla f|^2]
\\
= & f^{- \tfrac{2}{n - 2}} [2 \lambda_1 + 2 (\mu + J (\nu)) \\
& \quad + (\tfrac{n}{n - 1} h^2 - 2 h\ensuremath{\operatorname{tr}}p + 2
\langle D h, \nu \rangle) + | \chi^0 |^2 + \tfrac{n - 1}{n - 2} f^{- 1} |
\nabla f|^2],
\end{align}
  where $\chi^0 = \chi - \tfrac{1}{n - 1} \theta g_{\Sigma}$ denotes the trace
  free part of the null second fundamental form.
  
  By the condition \eqref{modfied dec}, we have that $\tilde{S} \geqslant 0$.
  If $\tilde{S} > 0$ somewhere, then it is well know that $\Sigma$ carries a
  metric of strict positive scalar curvature. If $\tilde{S}$ vanishes
  identically, then $\lambda_1 = 0$, $\chi^0$ vanishes along $\Sigma$, $\theta
  = h$ along $\Sigma$,
  \[ 2 (\mu + J (\nu)) + (\tfrac{n}{n - 1} h^2 - 2
     h\ensuremath{\operatorname{tr}}p + 2 \langle D h, \nu \rangle) \]
  vanishes along $\Sigma$ and $f$ is a positive constant. We see from
  \eqref{conformal related S} that $(\Sigma, h)$ is of vanishing scalar
  curvature. By a result of Bourguinon (see {\cite{kazdan-prescribing-1975}}),
  $\Sigma$ carries a metric of positive scalar curvature unless $\Sigma$ is
  Ricci flat.
\end{proof}

Eichmair {\cite{eichmair-plateau-2009}} and Andersson-Metzger
{\cite{andersson-area-2009}} established the existence of a marginally outer
trapped surface in domains bounded by two boundaries with certain convexity
conditions. In addition, Eichmair obtained the solution to the Plateau problem
via Perron methods. We note that the Eichmair's approach works for surfaces of
prescribed null expansion with only minor modifications. In particular analogs
of Theorem 1.1 and Theorem 3.1 of {\cite{eichmair-plateau-2009}} could be
established via a study of the Jang equation,
\begin{equation}
  \ensuremath{\operatorname{div}}_g (\tfrac{D u}{\sqrt{1 + |D u|^2}}) + (g^{i
  j} - \tfrac{D^i u D^j u}{1 + |D u|^2}) p_{i j} - \tau u - h = 0,
\end{equation}
here $\tau \geqslant 0$ is a constant. The following theorem is an analog of
{\cite[Theorem 3.1]{eichmair-plateau-2009}}. We leave the proof and the
solution of the Plateau problem to interested readers.

\begin{theorem}
  Let $(M^n, g, p)$ be an initial data set and let $U$ be a connected bounded
  open subset with smooth embedded boundary $\partial U$. Assume that the
  boundary consists of two non-empty closed hypersurfaces $\partial_1 U$ and
  $\partial_2 U$ so that
  \begin{equation}
    H_{\partial_1 U} + (\ensuremath{\operatorname{tr}}_{\partial_1 U} p - h) >
    0 \text{ and } H_{\partial_2 U} -
    (\ensuremath{\operatorname{tr}}_{\partial_2 U} p - h) > 0,
  \end{equation}
  where the mean curvature is computed with respect to the normal pointing out
  of $U$. Then there exists a closed boundary $\Sigma^{n - 1}$ homologous to
  $\partial_1 U$ which is $C$-almost minimizing in $U$ for a constant $C = C
  (|p|_{C (\bar{U})}, |h|_{C (\bar{U})})$. (Hence, $\Sigma$ has a singular set
  of Hausdorff dimension less than $n - 8$ which satisfies \eqref{pne}
  distributionally. In particular, if $2 \leqslant n \leqslant 7$, $\Sigma$ is
  non-empty, closed and embedded.) If $3 \leqslant n \leqslant 7$ and if the
  data satisfies \eqref{modfied dec}, the $\Sigma$ has non-negative Yamabe
  type.
\end{theorem}

\text{{\bfseries{Acknowledgment}}} Research of Xiaoxiang Chai is supported by
KIAS Grants under the research code MG074402.

\end{document}